\documentclass[10pt,reqno]{amsart}
\usepackage{amsmath, amsthm, amsfonts,amscd,epsfig,amssymb}
\usepackage{latexsym}
\usepackage[T1]{fontenc}
\usepackage{url}

\theoremstyle{plain}
   \newtheorem{theorem}{Theorem}[section]
   \newtheorem{proposition}[theorem]{Proposition}
   \newtheorem{lemma}[theorem]{Lemma}
   \newtheorem{corollary}[theorem]{Corollary}
   \newtheorem{conjecture}[theorem]{Conjecture}
   
\theoremstyle{definition}
   
   \newtheorem{question}{Question}

\theoremstyle{remark}

\author[P.~Br\"and\'en]{Petter Br\"and\'en}
\address{Department of Mathematics, 
Stockholm University, 
SE-106 91 Stockholm, Sweden}
\email{pbranden@math.su.se}  
 \thanks{The author is a Royal Swedish Academy of Sciences Research Fellow 
supported by a grant from the Knut and Alice Wallenberg Foundation.}
\thanks{To appear in Advances in Mathematics.}


\newcommand{\NN}{\mathbb{N}}

\newcommand{\JJ}{\mathcal{J}}
\newcommand{\MM}{\mathcal{M}}

\newcommand{\BB}{\mathcal{B}}
\newcommand{\VV}{\mathcal{V}}
\newcommand{\EE}{\mathcal{E}}

\newcommand{\ZZ}{\mathbb{Z}}

\newcommand{\RR}{\mathbb{R}}
\newcommand{\CC}{\mathbb{C}}

\renewcommand{\Im}{{\rm Im}}

\def\newop#1{\expandafter\def\csname #1\endcsname{\mathop{\rm
#1}\nolimits}}

\newop{diag}
\newop{supp}
\newop{per}
\newop{rank}
\newop{rk}
\newop{Sym}
\newop{Int}
\newop{glb}
\newop{Trp}
\newop{span}

\begin{document}
\title[Obstructions to determinantal representability ]
{Obstructions to determinantal representability }
\begin{abstract}
There has recently been ample interest in the question of which sets can be represented by linear matrix inequalities (LMIs). A necessary condition is that the set is rigidly convex, and it has been conjectured that rigid convexity is also sufficient. To this end Helton and Vinnikov conjectured that any real zero polynomial admits a determinantal representation with symmetric matrices. We disprove this conjecture. By relating the question of finding LMI representations to the problem of determining whether a polymatroid is representable over the complex numbers, we find a real zero polynomial such that no power of it admits 
a determinantal representation. The proof uses recent results of Wagner and Wei on matroids with the half-plane property, and the polymatroids associated to hyperbolic polynomials introduced by Gurvits. 
\end{abstract}
\maketitle
\section{Representing sets with linear matrix inequalities}
Motivated by powerful techniques commonly used in control theory, there has recently been considerable interest in the following question. 
\begin{question}\label{Q}
Which subsets of $\RR^n$ can be represented by linear matrix inequalities (LMIs)? That is, which sets $\mathcal{Y}$ are of the form 
\begin{equation}\label{lmi}
\mathcal{Y}=\{ (x_1,\ldots, x_n) \in \RR^n : A_0+A_1x_1+\cdots+A_nx_n \mbox{ is positive semidefinite}\}, 
 \end{equation}
 where $A_0,\ldots, A_n$ are real symmetric $m \times m$ matrices?
\end{question}  
 In two variables such sets were 
characterized by Helton and Vinnikov \cite{HV} by so called rigidly convex sets, thereby answering a question posed by Parrilo and Sturmfels \cite{PS}. We will always assume that $0$ is in the interior of $\mathcal{Y}$ and then the existence of a LMI representation of $\mathcal{Y}$ is equivalent to the existence  of a \emph{monic} LMI representation, i.e., a representation in which $A_0$ is the identity matrix. 

A polynomial $p \in \RR[x_1,\ldots, x_n]$ is a \emph{real zero polynomial} (RZ polynomial) if for each $x \in \RR^n$ and $\mu \in \CC$
\begin{equation}\label{RZ}
p(\mu x)=0 \quad \mbox{ implies } \quad \mu \mbox{ is real. }
\end{equation}
A set $\mathcal{Y} \subseteq \RR^n$ is \emph{rigidly convex} (at the origin) if there is a RZ polynomial $p$ for  which 
$\mathcal{Y}$ is equal to the closure of the connected  component of 
$$\{ x \in \RR^n: p(x)>0\}$$
containing the origin.

In what follows $I$ will always denote the identity matrix of appropriate size. It is not hard to see that if $A_1, \ldots, A_n$ are symmetric or hermitian matrices of the same size then the polynomial 
$\det(I+x_1A_1+\cdots+x_nA_n)$ is a RZ polynomial. In two variables Helton and Vinnikov provided a converse to this fact. 
\begin{theorem}[Helton--Vinnikov, \cite{HV}]\label{HV}
Let $p(x,y)$ be a RZ polynomial of degree $d$, and suppose that $p(0,0)=1$. Then there are symmetric matrices $A$ and $B$ of size $d\times d$ such that 
$$
p(x,y)= \det(I+xA+yB).
$$
\end{theorem}
Theorem \ref{HV} also settles a conjecture of Peter Lax which asserts that any hyperbolic degree $d$ polynomial in three variables can be represented by a determinant, see \cite{LPR}. By a simple count of parameters one sees that the exact analog 
of Theorem \ref{HV} in three or more variables fails. However, the count of parameters does not preclude a determinantal 
representation with matrices of a size larger than the degree. To this end, Helton and Vinnikov  \cite[p. 668]{HV} made the following conjecture.
\begin{conjecture}[Helton--Vinnikov]\label{con1}
Let $p(x_1,\ldots, x_n)$ be a RZ polynomial, and suppose that $p(0)=1$. Then there 
are symmetric matrices $A_1,\ldots, A_n$ such that 
$$
p(x_1,\ldots, x_n)= \det(I+x_1A_1+\cdots+x_nA_n).
$$
\end{conjecture}
In Section \ref{para} we will find a family of counterexamples to Conjecture \ref{con1}. The following  relaxation of Conjecture \ref{con1} has also been suggested. 
 \begin{conjecture}\label{con2}
Let $p(x_1,\ldots, x_n)$ be a RZ polynomial, and suppose that $p(0)=1$. Then there 
are symmetric matrices $A_1,\ldots, A_n$, and a positive integer $N$ such that 
$$
p(x_1,\ldots, x_n)^N= \det(I+x_1A_1+\cdots+x_nA_n).
$$
\end{conjecture}
In Section \ref{poly} we find a counterexample to Conjecture \ref{con2} by relating the problem of finding determinantal 
representations to the problem of determining whether a given polymatroid is representable (comes from a subspace arrangement). The counterexample arises from the fact that the V\'amos cube is a matroid that has the so called 
\emph{half-plane property} but is not representable over any field, see \cite{WW}. 

The conjecture that any rigidly convex set can be represented by LMIs still remains open.

\section{Counterexamples by a count of parameters}\label{para} 
A homogeneous polynomial $h(x_1,\ldots,x_n) \in \RR[x_1, \ldots, x_n]$ is \emph{hyperbolic} with respect to $e \in \RR^n$ if  $h(e) \neq 0$ and if for each $x \in \RR^n$ and $\mu \in \CC$
\begin{equation}\label{RZ}
h(x+\mu e)=0 \quad \mbox{ implies } \quad \mu \mbox{ is real, }
\end{equation}
see \cite{Ga,Ren}. 
Clearly, if $h$ is hyperbolic with respect to $e$, then $h(x+e)$ is a RZ polynomial. 
The \emph{hyperbolicity cone} of $h$ at $e$  is the set of all $x \in \RR^n$ for which the univariate polynomial 
$t \mapsto p(x+te)$ has only negative zeros. 

We will use the Cauchy--Binet theorem. Let $[n]=\{1,\ldots, n\}$ and let $\binom {[n]} k$ denote the set of all $k$-element subsets of $[n]$. 
If $A$ is an $n \times m$ matrix and $S \subseteq [n]$, $T \subseteq [m]$ are two sets of the same size we denote by $A(S,T)$ the minor of $A$ 
with rows indexed by $S$ and columns indexed by $T$. 
\begin{theorem}[Cauchy--Binet]
Let $A$ be an $m \times n$ matrix and $B$ an $n \times m$ matrix. Then 
$$
\det(AB) = \sum_{S \in \binom {[n]} m} A([m], S)B(S,[m]).
$$
\end{theorem}

\begin{theorem}\label{decrease}
Let $h(x) \in \RR[x_1,\ldots, x_n]$ be a hyperbolic polynomial with respect to $e$, and let 
$p(x)$ be the RZ-polynomial defined by $p(x)=h(x+e)$. If $p$ admits a representation 
$$
p(x)= \det(I+x_1A_1+ \cdots + x_nA_n)
$$
where $A_j$ is symmetric (hermitian) and of size $N\times N$ for all $j$, then $p$ admits a representation 
 $$
p(x)= \det(I+x_1B_1+ \cdots + x_nB_n)
$$
where $B_j$ is symmetric (hermitian) and of size $d\times d$ for all $j$, and $d$ is the degree of $h$ and $p$. 
\end{theorem}
\begin{proof}
By considering a linear change of variables 
we may, and will,  assume that 
$h(x)$  has hyperbolicity cone containing  $\RR_{+}^d$, where $\RR_+$ is the set of positive reals,  and that $e=(1,\ldots,1)^T =: \mathbf{1}$. 

We claim that $A_j$ and $I-\sum_{j=1}^nA_j$ are positive semidefinite (PSD) for all $j$. The univariate polynomial $$\det(I+tA_j)= p(0,\ldots, t, \ldots, 0)=h(\mathbf{1}+t\delta_j),$$ where  $\delta_j$ is the $j$th standard bases vector, has only are real and non-positive zeros (since $\delta_j$ is in the closure of the hyperbolicity cone of $h$). Hence $A_j$ is PSD. Similarly 
\begin{eqnarray*}
\det\left(tI+I-\sum_{j=1}^nA_j\right) &=& (1+t)^Np(-1/(1+t),\ldots, -1/(1+t))\\ 
&=& (1+t)^Nh(1-1/(1+t),\ldots,1 -1/(1+t)) \\
&=& (1+t)^{N-d}t^d.
\end{eqnarray*}
Hence $I-\sum_{j=1}^nA_j$ is PSD of rank $N-d$. Suppose that $A_i$ has rank $r_i$. Write 
$A_i$ as $A_i = \sum_{j=1}^{r_i}A_{ij}$, where $A_{ij}=v_{ij}v_{ij}^*$ is PSD of rank $1$, $v_{ij}^*$ is the hermitian adjoint of $v_{ij}$, and $v_{ij} \in \CC^N$. Similarly 
we may write $I- \sum_{j=1}^nA_j$ as a sum $\sum_{j=1}^{N-d}C_{j}$, where $C_{j}=u_{j}u_{j}^*$ is PSD of rank $1$. 
Let now 
$$
P(\tilde{x}, \tilde{y}) = \det\left(\sum_{j=1}^{N-d}C_{j}y_j+ \sum_{i,j}A_{ij}x_{ij}\right),  
$$
where $\tilde{x}= (x_{ij})_{i,j}$ and $\tilde{y}=(y_1, \ldots, y_{N-d})$ are new variables.  Let $B$ be the matrix 
with columns $u_1, \ldots, u_{N-d}, v_{11}, \ldots, v_{1r_1}, \ldots v_{nr_n}$. Rename the columns and variables so that $B=[w_1,\ldots, w_M]$, and the corresponding variables are $z=(z_1,\ldots, z_M)$.  By construction and the Cauchy-Binet theorem 
$$
P(z)= \det(B Z B^*)= \sum_{S \in \binom {[M]} N} |B(S)|^2 \prod_{j \in S}z_j, 
$$
where $Z= \diag(z_1,\ldots, z_M)$, and $B(S)=B([N],S)$ is the $N \times N$ minor of $B$ with columns indexed by $S$. 

We obtain $h(x)$ from $P(z)$ by setting $y_j=1$ and $x_{ij}=x_i$ for all $i$ and $j$. Since all coefficients of 
$P(z)$ are nonnegative and $h(x)$ is homogeneous of degree $d$ we have that for each $S$ with $B(S) \neq 0$ there 
are precisely $d$ indices that correspond to $x$-variables and $N-d$ variables that correspond to $y$-variables.   
This means that $U\cap V=(0)$, where  $U= \span\{u_i\}_{i=1}^{N-d}$ and $V= \span\{v_{ij}: 1\leq i \leq n \mbox{ and } {1\leq j \leq r_i} \}$. Hence we 
may write $B$ as $B=PM$ where $P$ is invertible and $M$ is a block matrix
$$M=\left[ \begin{array}{cc}
M_1 & 0 \\
0 & M_2 \end{array} \right],
$$
where $M_1$ has $N-d$ columns and $M_2$ has $M+d-N$ columns. Thus 
\begin{eqnarray*}
P(z)&=&\det(PP^*)\det(MZM^*)\\
&=&\det(PP^*)\sum_{S_1 \in \binom {[N-d]}{N-d}}\sum_{S_2 \in \binom {[N-d+1, M]} d} |M_1(S_1)|^2|M_2(S_2)|^2 \prod_{j \in S_1}z_j \prod_{j \in S_2}z_j \\
&=& \det(PP^*)|\det(M_1)|^2y_1\cdots y_{N-d} \sum_{S_2 \in \binom {[N-d+1, M]} d} |M_2(S_2)|^2 \prod_{j \in S_2}z_j \\
&=&  \det(PP^*)|\det(M_1)|^2y_1\cdots y_{N-d} \det\left( \sum_{i=1}^{M-n+d} z_{N-d+i}m_im_i^*\right),
\end{eqnarray*}
where $m_i$ is the $i$th column of $M_2$. Setting $y_i=1$ and $x_{ij}=x_i$ for all $i$ and $j$ we obtain a 
representation 
$$
h(x)=\det(PP^*)|\det(M_1)|^2\det\left(\sum_{i=1}^n T_i x_i\right),
$$
where each $T_i$ is PSD of size $d \times d$, and $\sum_{i=1}^n T_i$ is positive definite. It follows that $p(x)$ has a representation of the desired form.  
\end{proof} 
 Nuij \cite{Nu} proved that the space of all hyperbolic polynomials of degree $d$ that are hyperbolic with respect to $e \in \RR^n$ has nonempty interior. Hence so does the space of RZ polynomials considered in Theorem \ref{decrease}. Since any such polynomial that admits a determinantal representation also admits a determinantal representation with matrices 
 of size $d$, a count of parameters provides counterexamples to  Conjecture \ref{con1}.

\section{Representability of polymatroids}\label{poly}
We will see here that Question \ref{Q} is closely related to the old problem of determining if a polymatroid is representable 
over $\CC$. 

An (integral) \emph{polymatroid} on a finite set $E$ is a function $r : 2^E \rightarrow \NN$ such that 
\begin{enumerate}
\item $r(\emptyset) = 0$; 
\item If $S \subseteq T \subseteq E$, then $r(S) \leq r(T)$; 
\item $r$ is \emph{submodular}, that is, 
$$
r(S \cup T)+ r(S \cap T) \leq r(S) + r(T), 
$$
for all subsets $S$ and $T$ of $E$. 
\end{enumerate}
A natural class of polymatroids arises from \emph{subspace arrangements}. Let $E$ be a finite set and 
 $\VV= (V_j)_{j\in E}$  a collection of subspaces of a 
finite dimensional vector space $V$ over a field $K$. Then the function $r_\VV : 2^E \rightarrow \NN$ defined by 
$$
r_\VV(S)= \dim\left(\sum_{i \in S}V_i\right), 
$$
where $\sum_{i \in S}V_i$ is the smallest subspace containing $\cup_{i \in S}V_i$,  
is a polymatroid. This follows from the dimension formula for subspaces  $$\dim(U+W)+\dim(U\cap W) = \dim(U)+\dim(W).$$
We say that a polymatroid, $r$, is \emph{representable over the field} $K$ if there is a subspace arrangement $\VV$ of subspaces of a vector space over $K$ such that $r=r_\VV$. There are several inequalities known to hold for representable polymatroids, see \cite{DFZ,Ing,Kinser}. The simplest of these are known as the 
\emph{Ingleton inequalities}. 
\begin{lemma}[Ingleton inequalities \cite{Ing}] 
Suppose that $\VV =(V_1,\ldots, V_n)$ is a subspace arrangement. Then 
\begin{eqnarray*}
 && r_\VV(S_1\cup S_2)+r_\VV(S_1\cup S_3 \cup S_4)+r_\VV(S_3)+r_\VV(S_4)+r_\VV(S_2\cup S_3 \cup S_4) \leq \\
 && r_\VV(S_1\cup S_3) +r_\VV(S_1\cup S_4)+r_\VV(S_2\cup S_3)+r_\VV(S_2\cup S_4)+r_\VV(S_3\cup S_4)
\end{eqnarray*}
for all $S_1, S_2, S_3, S_4 \subseteq [n]$. 
\end{lemma}

To see that subspace arrangements over $\CC$ or $\RR$ are closely related to determinantal representability we proceed to express the rank function in terms of determinants.  
Suppose that $A_1, \ldots, A_n$ are positive semidefinite matrices of the same size $m$, and let 
$\VV= (V_1, \ldots, V_n)$ be the subspace arrangement in $\CC^m$ defined by letting $V_i$ be the image of 
$A_i$ for all $i$. Then 
$$
r_\VV(S)= \rank\left(\sum_{i \in S} A_i\right)= \deg\left( \det\left(I+t\sum_{i \in S}A_i\right)\right), 
$$
for all $S \subseteq [n]$. 
To see this it is enough (by spectral decomposition) to consider the case when all matrices are of rank one and that $S=[n]$. Write $A_i$ as $A_i=v_iv_i^*$ where 
$v_i \in \CC^m$, and let $D$ be the 
$(m+n)\times (m+n)$ diagonal matrix with the first $m$ entries equal to one and the remaining entries equal to $t$. Let further $B$ be the  $m \times (m+n)$ matrix with columns $\delta_1,\ldots, \delta_m, v_1,\ldots, v_n$, where $\delta_i$ is the $i$th standard bases vector of $\CC^m$. Then by the Cauchy--Binet theorem 
$$
\det\left(I+t\sum_{i}A_i\right)= \det(BDB^*)= \sum_{S \in \binom {[m+n]} m} |B(S)|^2 t^{|S \cap \{m+1, \ldots, m+n\}|}.
$$
Hence the degree of the above polynomial is the size of a maximal linearly independent subset of $\{v_1,\ldots, v_n\}$, 
that is, the dimension of $V_1 + \cdots + V_n$. 

Next we will see how polymatroids arise from hyperbolic polynomials. This connection was observed by Gurvits \cite{Gu}. If $h(x_1,\ldots, x_n)$ is a hyperbolic polynomial with respect to $e$, we define a rank function 
$\rank_h : \RR^n \rightarrow \NN$ by 
$$
\rank_h(x)= \deg(h(e+xt)).
$$
The rank does not depend on the choice $e$, but only on the hyperbolicity cone of $h$, that is, $\deg(h(e+xt))=\deg(h(e'+xt))$ for all 
$e'$ in the hyperbolicity cone containing $e$, see \cite{Gu} and Section \ref{hyprank}.

The next proposition follows from the work of Gurvits \cite{Gu}. He uses Theorem \ref{HV}. In Section \ref{hyprank} we give a proof that does not rely on the Lax conjecture. 
\begin{proposition}\label{rankp}
Let $h \in \RR[x_1,\ldots, x_m]$ be a hyperbolic polynomial with respect to $e \in \RR^m$,  and let $\EE = (e_1, \ldots, e_n)$ be a tuple of $n$ vectors lying in the closure of the hyperbolicity cone of $h$ containing $e$. Then 
the function $r_\EE : 2^{[n]} \rightarrow \NN$ defined by 
$$r_\EE(S)= \rank_h\left(\sum_{i \in S}e_i\right) $$ is a polymatroid. 
\end{proposition}

 A \emph{matroid}, $\MM$,  may be defined as a polymatroid for which the rank function satifies $r_\MM(\{i\})\leq1$ for all $i\in E$.  Let $\MM$ be a matroid on $E$. The \emph{set of bases} of $\MM$  is 
$$
\BB(\MM)= \{ S \subseteq E : |S|=r_\MM(S)=r_\MM(E)\}. 
$$
It follows from the equivalent definitions of matroids, see \cite{Oxley}, that 
\begin{equation}\label{eqrank}
r_\MM(S)= \max\{ |S\cap B| : B \in \BB(\MM)\}, 
\end{equation}
for all $S \subseteq E$. 
The \emph{bases generating polynomial} of $\MM$ is  the polynomial in the variables 
$(x_i)_{i \in E}$ defined by 
$$
h_\MM(x)= \sum_{S \in \BB(\MM)} \prod_{j \in S}x_j. 
$$
For $i \in E$, let $\delta_i \in \RR^E$ be defined by $\delta_j(i)=\delta(i,j)$, where $\delta(i,j)$ is the Kronecker delta. By \eqref{eqrank} 
\begin{equation}\label{marank}
r_\MM(S)= \deg \left( h_\MM\left(\mathbf{1}+t\sum_{i \in S}\delta_i \right)\right),
\end{equation}
for all $S \subseteq E$.

Let $V_8$ be the V\'amos cube, see \cite{Oxley}. The set of bases of $V_8$ are all subsets of size four in Fig. \ref{tru}, that do not lie in an affine  plane. The V\'amos cube is not representable over any field. However, its bases generating polynomial is hyperbolic with hyperbolicity cone containing $\RR_+^8$. This follows from the fact that $V_8$ is a so called \emph{half-plane property matroid} (see \cite{COSW}) which was proved by Wagner and Wei \cite{WW}. 
\begin{figure}[htp] 
 \centering
 \includegraphics[height=2in]{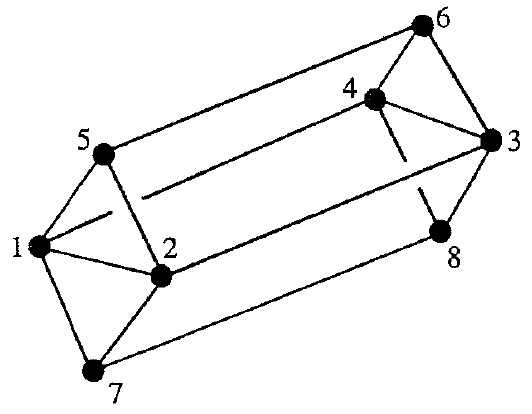}
\caption{\label{tru} The Vamos cube.}
\end{figure}

We are now in a position to establish the counterexample to Conjecture \ref{con2}. 
\begin{theorem}
Let $p(x)=h_{V_8}(x_1+1, \ldots, x_8+1)$. Then 
\begin{enumerate}
\item $p(x)$ is a RZ-polynomial; 
\item There is no positive  integer $N$ such that $p(x)^N$ has a determinantal representation. 
\end{enumerate}
\begin{proof}
Wagner and Wei \cite{WW} proved that  $h_{V_8}(x)$ is a stable polynomial, that is, $h_{V_8}(x)$ is non-zero whenever $\Im(x_i) >0$ for all $i$. Hence, 
if $x \in \RR^8$ and $y \in \RR_+^8$, then the polynomial $h_{V_8}(x+ty)$ has only real zeros. Thus $h_{V_8}(x)$ is hyperbolic with hyperbolicity cone containing $\RR_+^8$. As previously noted it follows that $p(x)=h_{V_8}(x+ \mathbf{1})$ is a RZ polynomial. 

Suppose that there is an integer $N>0$ for which 
$$
p(x)^N= \det(I +x_1A_1+\cdots+ x_8A_8), 
$$
where $A_i$ is hermitian for all $i$. As in the proof of Theorem \ref{decrease} it follows that 
$$
h_{V_8}(x)^N= \det(x_1B_1+\cdots+ x_8B_8),
$$
where $B_i$ is positive semidefinite of size 
$(8N)\times (8N)$ for all $i$. Of course $h_{V_8}(x)^N$ is also hyperbolic with the same hyperbolicity cone as $h_{V_8}(x)$.  By \eqref{marank}, the rank function of $h_{V_8}(x)^N$ with respect to  $\EE= \{\delta_1,\ldots, \delta_8\}$ satisfies
$$
r_{\EE}(S) = \deg \left( h_{V_8}^N\left(\mathbf{1}+t\sum_{i \in S}\delta_i \right)\right)=Nr_{V_8}(S), 
$$
for all $S \subseteq [8]$. Hence there is a subspace arrangement $\VV=(V_1, \ldots, V_8)$ for which 
$$
r_\VV= Nr_{V_8}. 
$$
However, it is known that $r_{V_8}$ (and thus also $N r_{V_8}$) fails to satisfy Ingleton's inequalities. This is seen by choosing 
$$
S_1=\{5,6\}, \quad S_2=\{7,8\}, \quad S_3=\{1,4\}, \quad S_4=\{2,3\},  
$$
in the Ingleton inequalities.   

\end{proof}

\end{theorem}

\section{Properties of the rank function of a hyperbolic polynomial }\label{hyprank}
For completeness we give proofs that do not use the Lax conjecture of the properties we use about the rank function associated to a hyperbolic polynomial. We show that these properties are simple consequences of known concavity properties of stable polynomials and discrete convex functions.  

A \emph{step} from $\alpha \in \ZZ^n$ to $\beta \in \ZZ^n$ is a vector $s \in \ZZ^n$ of unit length such that 
$$
|\alpha+s-\beta| < |\alpha-\beta|, 
$$
where $|\alpha|= \sum_{i=1}^n |\alpha_i|$.  If $s$ is a step from $\alpha$ to $\beta$ we write $\alpha \stackrel s {\rightarrow} \beta$. A set $\JJ \subseteq \ZZ^n$ is called a {\em jump system} if it respects the following axiom.  
\begin{itemize}
\item[(J):] If 
$\alpha,\beta \in \JJ$, $\alpha \stackrel s {\rightarrow} \beta$ and  $\alpha+s \notin \JJ$, then there is 
a step $t$ such that $\alpha+s \stackrel t {\rightarrow} \beta$ and $\alpha+s+t \in \JJ$. 
\end{itemize}
The support, $\supp(p)$, of a polynomial $p(x)= \sum_{\alpha \in \NN^n} a(\alpha) x_1^{\alpha_1}\cdots x_n^{\alpha_n}$ is the set $\{\alpha \in \NN^n : a(\alpha) \neq 0\}$. A polynomial $p \in \CC[x_1,\ldots, x_n]$ is \emph{stable} if 
$p(x) \neq 0$ whenever $\Im(x_j)>0$ for all $j$. Let $\leq$ be the usual product order on $\ZZ^n$, i.e., $\alpha \leq \beta$ if $\alpha_j \leq \beta_j$ for all $j$.
\begin{theorem}[\cite{Br1}]\label{supp}
The support of a stable polynomial is a jump system. 

Moreover, if all the Taylor coefficients of the stable polynomial $p$ are nonnegative, and $\alpha, \beta \in \supp(p)$ with $\alpha \leq \beta$, then  $\gamma \in \supp(p)$ for all $\alpha \leq \gamma \leq \beta$. 
\end{theorem}
We need the following simple property of jump systems. 
\begin{lemma}\label{simple}
 If 
$\JJ \subset \ZZ^n$ is a finite jump system and $\alpha, \beta \in \JJ$ are maximal (or minimal) with respect to $\leq$, then $|\alpha|= |\beta|$. 
\end{lemma}
\begin{proof}
The proof is by contradiction. Let $M$ be the set of  maximal elements  $\beta$ of $\JJ$, with $|\beta|=d$ maximal. Suppose further that $\beta \in M$ is of minimal $L^1$-distance  
to the set of all maximal (w.r.t. $\leq$) elements $\alpha$ with $|\alpha|<d$. Let $\alpha$ be a maximal element that 
realizes the above distance to $\beta$.  

 Clearly 
$\alpha_j > \beta_j$ for some $j$. Thus $\delta_j$ is a step from $\beta$ to $\alpha$ and $\beta + \delta_j \notin \JJ$. By 
(J), $\beta'=\beta + \delta_j + s \in \JJ$ for some step $s$ from $\beta + \delta_j$ to $\alpha$. Since $\beta$ is maximal, the nonzero coordinate in $s$ is negative. Now, $|\beta'|=|\beta|$, so $\beta'$ is maximal (w.r.t. $\leq$). However, 
$|\beta'-\alpha|<|\beta -\alpha|$ which is the desired contradiction. 
\end{proof}

\begin{lemma}\label{stablehyp}
Suppose that $h$ is hyperbolic with respect to $e \in \RR^n$ and that $e_1, \ldots, e_m$ lie in the hyperbolicity cone of $h$, and $e_0 \in \RR^n$. Then 
the polynomial 
$$
p(x_1,\ldots, x_m) = h\left(e_0 + \sum_{j=0}^m e_jx_j\right) 
$$
is stable or identically zero. 

Moreover if additionally $h(e)>0$ and $e_0$ is in the closure of the hyperbolicity cone of $e$, then all Taylor coefficients of $p$ are nonnegative. 
\end{lemma}  
\begin{proof}
By Hurwitz' theorem we may assume that $e_1,\ldots, e_m$ are in the hyperbolicity cone containing $e$. Assume 
that $\alpha \in \RR^m$ and $\beta \in \RR_+^m$. Then 
$$
p(\alpha+i\beta)=h\left(e_0 +\sum_{j=1}^m\alpha_je_j+i\left(\sum_{j=1}^m\beta_je_j\right)\right) \neq 0, 
$$
since the hyperbolicity cone  is convex, see \cite{Ga, Ren}. Thus $p$ is stable. 

To prove the last statement we  show that 
all the Taylor coefficients of $$q(x_0,\ldots,x_m)= h(x_0e_0+\cdots +x_me_m)$$ are nonnegative. Clearly $q$ is hyperbolic (or identically zero) with hyperbolicity cone containing $\RR_+^d$, or equivalently, $q$ is homogeneous and stable. It is not hard to prove that such polynomials have nonnegative Taylor coefficients, either using Renegar derivatives \cite{Ren}, or as in \cite{BB,COSW}.  Hence, the Taylor coefficients of $p$ are nonnegative. 
\end{proof}

\begin{lemma}\label{diff}
Suppose that $h$ is hyperbolic with respect to $e \in \RR^n$ and that $e'$ lies in the hyperbolicity cone containing $e$. Then 
$$
\deg(h(e+xt))= \deg(h(e'+xt))
$$
for all $x \in \RR^n$. 
\end{lemma}
\begin{proof}
The polynomial $p(s,t)= h(x+se+te')$ is stable by Lemma \ref{stablehyp}. 
Let the degree of $h$ be $d$. By Theorem \ref{supp}, $\JJ= \supp(p)$ is a jump system and by Lemma~\ref{simple}
$$
\deg(h(e+sx))= \deg( s^dp(s^{-1},0)) =d - \min\{ i : (i,0) \in \JJ\}= d-\min\{|\alpha| : \alpha \in \JJ\},
$$
which does not depend on $e$. 
\end{proof}
\begin{corollary}\label{desc}
Let $h$ be a hyperbolic polynomial with respect to $e \in \RR^m$,  and let $\EE = (e_1, \ldots, e_n)$ be a tuple of $n$ vectors lying in the closure of the hyperbolicity cone of $h$ containing $e$. Let further $\JJ$ be the support of the stable and homogeneous polynomial $h(x_1e_1+ \cdots + x_ne_n)$. Then 
the rank function associated to $\EE$ satisfies 
$$
r_\EE(S)= \max\left\{ \sum_{i \in S}\alpha_i : \alpha \in \JJ\right\}, 
$$
for all $S \subseteq [n]$. 
\end{corollary}
\begin{proof}
Let $p(x_1,\ldots, x_n)= h(e+x_1e_1+\cdots+x_ne_n)$. By Lemma \ref{stablehyp} $p$ is stable and has nonnegative Taylor coefficients. Hence 
$$
r_\EE(S)= \deg p\left(t \sum_{i\in S}e_i\right)= \max\left\{ \sum_{i \in S}\alpha_i :  \sum_{i \in S}\alpha_i\delta_i\in \supp(p)\right\}. 
$$
Note that the set of maximal elements (w.r.t. $\leq$) of $\supp(p)$ is equal to $\JJ$. Thus the inequality  $r_\EE(S)\leq  \max\left\{ \sum_{i \in S}\alpha_i : \alpha \in \JJ\right\}$ follows from Lemma \ref{simple}. Suppose that $\alpha \in \JJ$ and $S \subseteq [n]$. Since $0 \in \supp(p)$ and 
$0 \leq \sum_{i \in S}\alpha_i\delta_i \leq \alpha$ we have by Theorem \ref{supp} that $\sum_{i \in S}\alpha_i\delta_i \in \JJ$. Hence 
$r_\EE(S)\geq  \max\left\{ \sum_{i \in S}\alpha_i : \alpha \in \JJ\right\}$. 
\end{proof}
We may now prove Proposition \ref{rankp}.
\begin{proof}[Proof of Proposition  \ref{rankp}]
Keep the  notation in the proof of Corollary \ref{desc}. Then $\JJ$ is a jump system for which all vectors have constant sum. Such jump systems are known to coincide with the set of integer points of integral base polyhedra, see \cite{Murota}. Clearly $r_\EE$ satisfies (1) and (2) of the definition of a polymatroid. The submodularity of  
$$
S \mapsto  \max\left\{ \sum_{i \in S}\alpha_i : \alpha \in \JJ\right\}
 $$
holds for every constant sum jump system, see \cite{Murota}. 
\end{proof}

\noindent 
\textbf{Acknowledgments.} It is a pleasure to thank Banff International Research Station for generous support during the workshop ``Convex Algebraic Geometry'' where this work was initiated. I also thank the participants of the workshop for stimulating discussions on the topic of this paper; especially Helton, Parrilo, Schweighofer, Sturmfels and Vinnikov.

\end{document}